\documentclass[10pt,twoside]{siamltex}
\usepackage{amsmath,amsfonts}
\usepackage{amsfonts,epsfig}
\usepackage{graphicx,psfrag}
\usepackage{hyperref}
\usepackage{verbatim,enumerate} 
\usepackage{mathrsfs}
\usepackage{color}
\usepackage{filecontents}
\usepackage{epstopdf}

\def\cvd{~\vbox{\hrule\hbox{%
     \vrule height1.3ex\hskip0.8ex\vrule}\hrule } }

\newtheorem{example}[theorem]{Example}

\newcommand\sgn{\operatorname{sgn}}
\newcommand\Eta{\mathrm{H}}
\newcommand\NumAdj{\mathrm{NumAdj}}

\begin{document}
\bibliographystyle{plain}

\title{Spectral Properties of Oriented Hypergraphs}
\author{Nathan Reff\thanks{Department of Mathematics, The College at Brockport: State University of New York,
Brockport, NY 14420, USA (\href{mailto:nreff@brockport.edu}{nreff@brockport.edu}).}}

\pagestyle{myheadings}
\markboth{N. Reff}{Spectral Properties of Oriented Hypergraphs}
\maketitle
\begin{abstract}
An oriented hypergraph is a hypergraph where each vertex-edge incidence is given a label of $+1$ or $-1$.  The adjacency and Laplacian eigenvalues of an oriented hypergraph are studied.  Eigenvalue bounds for both the adjacency and Laplacian matrices of an oriented hypergraph which depend on structural parameters of the oriented hypergraph are found.  An oriented hypergraph and its incidence dual are shown to have the same nonzero Laplacian eigenvalues.  A family of oriented hypergraphs with uniformally labeled incidences is also studied.  This family provides a hypergraphic generalization of the signless Laplacian of a graph and also suggests a natural way to define the adjacency and Laplacian matrices of a hypergraph.  Some results presented generalize both graph and signed graph results to a hypergraphic setting.
\end{abstract}

\begin{keywords} 
Oriented hypergraph, hypergraph Laplacian, hypergraph adjacency matrix, hypergraph Laplacian eigenvalues, signless Laplacian, signed graph, hypergraph spectra 
\end{keywords}

\begin{AMS}
05C50, 05C65, 05C22
\end{AMS}

\section{Introduction}
There have been several approaches to studying eigenvalues of matrices associated to uniform hypergraphs \cite{MR1235565,MR1405722,MR1325271,MR2842309}.  More recently, Cooper and Dutle have developed a hypermatrix approach to studying the spectra of uniform hypergraphs \cite{MR2900714}.  Rodr\'{i}guez developed a version of the adjacency and Laplacian matrices for hypergraphs without a uniformity requirement on edge sizes \cite{MR1890984}.  The work presented here does not require uniformity either, but is focused on hypergraphs with additional structure called \emph{oriented hypergraphs}.

An oriented hypergraph is a hypergraph where each vertex-edge incidence is given a label of $+1$ or $-1$.  This incidence structure can be viewed as a generalization of an oriented signed graph \cite{MR1120422}.  In \cite{ReffRusnak1} the author and Rusnak studied several matrices associated with an oriented hypergraph.  In this paper we study the eigenvalues associated to the adjacency and Laplacian matrices of an oriented hypergraph.

The paper is organized in the following manner.  In Section \ref{BackgroundSection}, a background on oriented hypergraphs and their matrices is provided.  In Section \ref{AdjEigenvaluesOH}, the adjacency matrix is further investigated.  Vertex-switching is shown to produce cospectral oriented hypergraphs.  Also, bounds for the spectral radius and eigenvalues of the adjacency matrix of an oriented hypergraph are derived.  In Section \ref{LapEigenvaluesOHLapEigenvaluesOH}, results on the Laplacian eigenvalues of an oriented hypergraph are established.  Vertex-switching is also shown to produce Laplacian cospectral oriented hypergraphs.  Although an oriented hypergraph and its dual are not always Laplacian cospectral, they have the same nonzero Laplacian eigenvalues.  A hypergraphic generalization of the signless Laplacian for graphs is mentioned, which provides an upper bound for the Laplacian spectral radius of an oriented hypergraph.  Bounds for the Laplacian eigenvalues of an oriented hypergraph that depend on both the underlying hypergraphic structure and the adjacency signatures are found.  In Section \ref{HypergraphSpectra}, a definition for the adjacency and Laplacian matrices of a hypergraph are stated. 

A consequence of studying oriented hypergraphs is that signed and unsigned graphs as well as hypergraphs can be viewed as specializations.  An oriented signed graph is an oriented hypergraph where all edges have size 2 (a 2-uniform oriented hypergraph).  This oriented signed graph has a natural edge sign associated to it, and hence a signed graph.  An unsigned graph can be thought of as a 2-uniform oriented hypergraph where all vertex-edge incidences are labeled $+1$.  This is not the only way to think of an unsigned graph, since other orientations may be more suitable in certain situations, although this is the simplest description.  Similarly, a hypergraph can be thought of as an oriented hypergraph where all vertex-edge incidences are labeled $+1$.  

\section{Background}\label{BackgroundSection}
\subsection{Oriented Hypergraphs}
A \emph{hypergraph} is a triple $H=(V,E,\mathcal{I})$, where $V$ is a set, $E$ is a set whose elements are subsets of $V$, and $\mathcal{I}$ is a multisubset of $V\times E$ such that if $(v,e)\in\mathcal{I}$, then $v\in e$.  Note that an edge may be empty.  The set $V$ is called the \emph{set of vertices}.  The set $E$ is called the \emph{set of edges}.  We may also write $V(H)$, $E(H)$ and $\mathcal{I}(H)$ for the set of vertices, edges and multiset of incidences of $H$, respectively.  Let $n:=|V|$ and $m:=|E|$.  If $(v,e)\in \mathcal{I}$, then $v$ and $e$ are \emph{incident}.  An \emph{incidence} is a pair $(v,e)$, where $v$ and $e$ are incident.  If $(v_i,e)$ and $(v_j,e)$ both belong to $\mathcal{I}$, then $v_i$ and $v_j$ are \emph{adjacent} vertices via the edge $e$.  The set of vertices adjacent to a vertex $v$ is denoted by $N(v)$.     

A hypergraph is \emph{simple} if for every edge $e$, and for every vertex $v\in e$, $v$ and $e$ are incident exactly once.  Unless otherwise stated, all hypergraphs in this paper are assumed to be simple.  A hypergraph is \emph{linear} if for every pair $e,f\in E$, $|e\cap f|\leq 1$.

The \emph{degree} of a vertex $v_i$, denoted by $d_i=\deg(v_i)$, is equal to the number of incidences containing $v_i$.  The \emph{maximum degree} is $\Delta:=\max_i d_i$.  The \emph{size} of an edge $e$ is the number of incidences containing $e$.   A $k$\emph{-edge} is an edge of size $k$.  A \emph{$k$-uniform hypergraph} is a hypergraph such that all of its edges have size $k$.

Given a hypergraph $H=(V(H),E(H),\mathcal{I}(H))$, there are several different substructures that can be created.  A \emph{subhypergraph} $S$ of $H$, denoted by $S=(V(S),E(S),\allowbreak\mathcal{I}(S))$, is a hypergraph with $V(S)\subseteq V(H)$, $E(S)\subseteq E(H)$ and $\mathcal{I}(S)\subseteq \mathcal{I}(H)\cap (V(S)\times E(S))$.  It is more common to define a subhypergraph as a hypergraph generated by a subset of the vertex set.  However, the definition above is more suitable for our purposes.  For a hypergraph $H=(V,E,\mathcal{I})$ with a vertex $v\in V$, the \emph{weak vertex-deletion} is the subhypergraph $H\backslash v=(V\backslash\{v\},E_v,\mathcal{I}_v)$, where 
\[E_v=\{e\cap (V\backslash\{v\}) : e\in E\},\] and 
\[\mathcal{I}_v=\mathcal{I}\cap \big((V\backslash\{v\})\times E_v\big).\]
Since edges are allowed to have size zero we do not need to add the additional condition $e\cap(V\backslash\{v\})\neq \emptyset$ to the definition of $E_v$, which is usually included in hypergraph literature.  Observe that edges incident to $v$ are not deleted in $H\backslash v$, as in the vertex-deletion of a graph.  That is why we call this type of deletion a weak vertex-deletion.  For a hypergraph $H=(V,E,\mathcal{I})$ with an edge $e\in E$, the \emph{weak edge-deletion} (or simply \emph{edge-deletion}), denoted by $H\backslash e$, is the subhypergraph $H\backslash e=(V,E\backslash \{e\},\mathcal{I}_e)$, where
\[ \mathcal{I}_e=\mathcal{I}\cap(V\times (E\backslash \{e\})).\]
The weak edge-deletion is the same as the graph version of edge-deletion.

The \emph{incidence dual} (or \emph{dual}) of a hypergraph $H=(V,E,\mathcal{I})$, denoted by $H^*$, is the hypergraph $(E,V,\mathcal{I}^*)$, where $\mathcal{I}^*:=\{(e,v):(v,e)\in\mathcal{I}\}$.  Thus, the incidence dual reverses the roles of the vertices and edges in a hypergraph.  

The set of size 2 subsets of a set $S$ is denoted by $\binom{S}{2}$.  The set of \emph{adjacencies} $\mathcal{A}$ of $H$ is defined as $\mathcal{A}:=\{(e,\{v_i,v_j\})\in E\times \binom{V}{2}: (v_i,e)\in \mathcal{I}\text{ and }(v_j,e)\in \mathcal{I}\}$.  We may also write $\mathcal{A}(H)$ for the set of adjacencies of $H$.  Observe that if $\{v_i,v_j\}\in \binom{V}{2}$, then the vertices $v_i$ and $v_j$ must be distinct.  Also, since $\mathcal{A}$ is a set there are no duplicate adjacencies.  The \emph{number of adjacencies containing vertex $v$} is denoted by $\NumAdj(v)$.  Observe that in general $d_j$, $|N(v_j)|$ and $\NumAdj(v_j)$ may all be different.  One must be careful of this fact when comparing similar graph and hypergraph bounds that will appear later in this paper.

The set of \emph{coadjacencies} $\mathcal{A}^*$ of $H$ is defined as $\mathcal{A}^*:=
\mathcal{A}(H^*)$.  We may also write $\mathcal{A}^*(H)$ for the set of coadjacencies of $H$.

An \emph{oriented hypergraph} is a pair $G=(H,\sigma)$ consisting of an \emph{underlying hypergraph} $H=(V,E,\mathcal{I})$, and an \emph{incidence orientation} $\sigma :\mathcal{I}\rightarrow\{+1,-1\}$.  Every oriented hypergraph has an associated \emph{adjacency signature} $\sgn:\mathcal{A}\rightarrow \{+1,-1\}$ defined by
\begin{equation}
\sgn(e,\{v_i,v_j\})=-\sigma (v_i,e)\sigma (v_j,e).
\end{equation}
Thus, $\sgn(e,\{v_i,v_j\})$ is called the \emph{sign} of the adjacency $(e,\{v_i,v_j\})$.  Instead of writing $\sgn(e,\{v_i,v_j\})$, the alternative notation $\sgn_{e}(v_i,v_j)$ will be used.  See Figure \ref{OHEx} for an example of an oriented hypergraph.

\begin{figure}[h!]
    \includegraphics[scale=0.6]{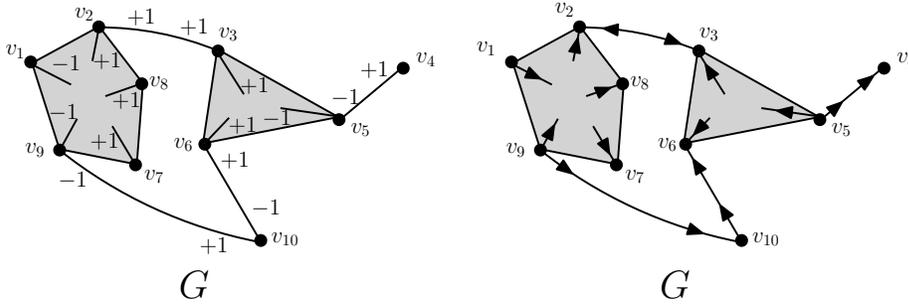}\centering
    \caption{A simple oriented hypergraph $G$ drawn in two ways.  On the left, the incidences are labeled with $\sigma$ values.  On the right, the $\sigma$ values assigned to the incidences are drawn using the arrow convention of $+1$ as an arrow going into a vertex and $-1$ as an arrow departing a vertex.}\label{OHEx}
\end{figure}

If $G=(H,\sigma)$ is an oriented hypergraph and $S=(V(S),E(S),\mathcal{I}(S))$ is a subhypergraph of $H$, then the \emph{oriented subhypergraph} $F$ of $G$ (generated by $S$) is defined by $F=(S,\sigma{\mid}_{\mathcal{I}(S)})$.  That is, the incidence orientation of $F$ is restricted to those incidences in $S$, and likewise, the adjacency signature of $F$ is restricted to those adjacencies of $S$.  The \emph{weak vertex-deletion} of $G$, denoted by $G\backslash v$, is the oriented subhypergraph $G\backslash v=(H\backslash v,\sigma{\mid}_{\mathcal{I}(H\backslash v)})$.  The \emph{weak edge-deletion} of $G$, denoted by $G\backslash e$, is the oriented subhypergraph $G\backslash e=(H\backslash e,\sigma{\mid}_{\mathcal{I}(H\backslash e)})$.

As with hypergraphs, an oriented hypergraph has an incidence dual.  The \emph{incidence dual} of an oriented hypergraph $G=(H,\sigma)$ is the oriented hypergraph $G^*=(H^*,\sigma^*)$, where the \emph{coincidence orientation} $\sigma ^{\ast }:\mathcal{I}^{\ast }\rightarrow \{+1,-1\}$ is defined by $\sigma ^{\ast }(e,v)\allowbreak=\sigma
(v,e)$, and the \emph{coadjacency signature} $\sgn^*:\mathcal{A}^*\rightarrow \{+1,-1\}$ is defined by
\begin{equation*}
\sgn^*(v,\{e_i,e_j\})=-\sigma^*(e_i,v)\sigma^*(e_j,v)=-\sigma (v,e_i)\sigma (v,e_j).
\end{equation*}

A \emph{vertex-switching function} is any function $\zeta:V\rightarrow \{-1,+1\}$.  \emph{Vertex-switching} the oriented hypergraph $G=(H,\sigma)$ means replacing $\sigma$ with $\sigma^{\zeta}$, defined by
\begin{equation}
\sigma^{\zeta}(v,e)=\zeta(v)\sigma(v,e);
\end{equation}
producing the oriented hypergraph $G^{\zeta}=(H,\sigma^{\zeta})$, with an adjacency signature $\sgn^{\zeta}$ defined by
\begin{align*}
\sgn_e^{\zeta}(v_i,v_j)&=-\sigma^{\zeta}(v_i,e)\sigma^{\zeta}(v_j,e)\\
&=-\zeta(v_i)\sigma(v_i,e)\sigma(v_j,e)\zeta(v_j)\\
&=\zeta(v_i)\sgn_e(v_i,v_j)\zeta(v_j).
\end{align*}  

We say two oriented hypergraphs $G_1$ and $G_2$ are \emph{vertex-switching equivalent}, written $G_1 \sim G_2$, when there exists a vertex-switching function $\zeta$, such that $G_2=G_1^{\zeta}$.  The equivalence class of $G$ formed under this relation is called a \emph{vertex-switching class}, and is denoted by $[G]$.

\subsection{Matrices and Oriented Hypergraphs}

Let $G$ be an oriented hypergraph.  The \emph{adjacency matrix} $A(G)=(a_{ij})\in \mathbb{R}^{n\times n}$ is defined by
\begin{equation*}
a_{ij}=
\begin{cases}
\displaystyle\sum_{e\in E}\sgn_{e}(v_{i},v_{j}) &\text{if $v_i$ is adjacent to $v_j$},\\
0 &\text{otherwise.}
\end{cases}
\end{equation*}
If $v_i$ is adjacent to $v_j$ , then
\begin{align*}
a_{ij}=\sum_{e\in E}\sgn_e(v_i,v_j)=\sum_{e\in E}\sgn_e(v_j,v_i)=a_{ji}.
\end{align*}
Therefore, $A(G)$ is symmetric.

Let $G=(H,\sigma)$ be a simple oriented hypergraph.  The \emph{incidence matrix} $\mathrm{H}(G)=(\eta _{ij})$ is the $n\times m$ matrix, with entries in $\{-1,0,+1\}$, defined by 
\begin{equation*}
\eta _{ij}=
\begin{cases} \sigma(v_{i},e_{j}) & \text{if }(v_{i},e_{j})\in \mathcal{I},\\
0 &\text{otherwise.}
\end{cases}
\end{equation*}

As with hypergraphs, the incidence matrix provides a convenient relationship between an oriented hypergraph and its incidence dual.  This is immediate by the definition of the incidence matrix and the incidence dual.

\begin{lemma}[\cite{ReffRusnak1},Theorem 4.1]\label{OHIncidenceMatrixDualTranspose}
If $G$ is an oriented hypergraph, then $\mathrm{H}(G)^{\text{T}}=\mathrm{H}(G^{\ast })$.
\end{lemma}

The \emph{degree matrix} of an oriented hypergraph $G$ is defined as $D(G):=\text{diag}(d_1,d_2,\allowbreak\ldots,d_n)$.  The \emph{Laplacian matrix} is defined as $L(G):=D(G)-A(G).$

The Laplacian matrix of an oriented hypergraph can be written in terms of the incidence matrix.

\begin{lemma}[\cite{ReffRusnak1}, Corollary 4.4]\label{OHLapIncidenceRelation}
If $G$ is a simple oriented hypergraph, then
\begin{enumerate}
\item $L(G)=D(G)-A(G)=\mathrm{H}(G)\mathrm{H}(G)^{\text{T}}$,
\item $L(G^{\ast })=D(G^{\ast})-A(G^{\ast })=\mathrm{H}(G)^T\mathrm{H}(G).$
\end{enumerate}
\end{lemma}

Vertex-switching an oriented hypergraph $G$ can be described as matrix multiplication of the incidence matrix, and as a similarity transformation of the adjacency and Laplacian matrices.  For a vertex-switching function $\zeta:V\rightarrow\{+1,-1\}$, we define a diagonal matrix $D(\zeta):=\text{diag}(\zeta(v_1),\zeta(v_1),\ldots,\zeta(v_n))$.  The following lemma shows how to calculate the switched oriented hypergraph's incidence, adjacency and Laplacian matrices.  

\begin{lemma}[\cite{ReffRusnak1}, Propositions 3.1 and 4.3]\label{OHLAHSwitchingSimilarityTrans} Let $G$ be an oriented hypergraph.  Let $\zeta$ be a vertex-switching function on $G$.  Then 
\begin{enumerate}
\item $\Eta(G^{\zeta})=D(\zeta)\Eta(G)$,
\item $A(G^{\zeta})=D(\zeta)^{\text{T}} A(G) D(\zeta)$, and
\item $L(G^{\zeta})=D(\zeta)^{\text{T}} L(G) D(\zeta)$.
\end{enumerate}
\end{lemma}

\subsection{Matrix Analysis}

Since the eigenvalues of any symmetric matrix $A\in\mathbb{R}^{n\times n}$ are real we will assume that they are labeled and ordered according to the following convention:
\[ \lambda_n(A) \leq \lambda_{n-1}(A) \leq \cdots \leq \lambda_2(A) \leq \lambda_1(A).\]

If $A\in \mathbb{R}^{n\times n}$ is symmetric, then the quadratic form $\mathbf{x}^{\text{T}}A\mathbf{x}$, for some $\mathbf{x}\in\mathbb{R}^n\backslash\{\mathbf{0}\}$, can be use to calculate the eigenvalues of $A$ using the following theorem. In particular, we can calculate the smallest and largest eigenvalues using the following, usually called the Rayleigh-Ritz Theorem.

\begin{lemma}[\cite{MR1084815},Theorem 4.2.2]\label{RRThm} Let $A\in\mathbb{R}^{n\times n}$ be symmetric.  Then
\begin{align*}
\lambda_1(A)&=\max_{\mathbf{x}\in \mathbb{R}^n \backslash \{\mathbf{0}\}} \frac{\mathbf{x}^{\text{T}}A\mathbf{x}}{\mathbf{x}^{\text{T}}\mathbf{x}} = \max_{\mathbf{x}^{\text{T}}\mathbf{x}=1} \mathbf{x}^{\text{T}}A\mathbf{x},\\
\lambda_n(A)&=\min_{\mathbf{x}\in \mathbb{R}^n \backslash \{\mathbf{0}\}} \frac{\mathbf{x}^{\text{T}}A\mathbf{x}}{\mathbf{x}^{\text{T}}\mathbf{x}} =\min_{\mathbf{x}^{\text{T}}\mathbf{x}=1} \mathbf{x}^{\text{T}}A\mathbf{x}.
\end{align*}
\end{lemma}

An $r\times r$ \emph{principle submatrix} of $A\in \mathbb{R}^{n\times n}$, denoted by $A_r$, is a matrix obtained by deleting $n-r$ rows and the corresponding columns of $A$.  The next lemma is sometimes called the the Cauchy Interlacing Theorem, or the inclusion principle.

\begin{lemma}[\cite{MR1084815}, Theorem 4.3.15]\label{interlacinglemma} 
Let $A\in\mathbb{R}^{n\times n}$ be symmetric and $r\in \{1,\ldots,n\}$.  Then for all $k\in \{1,\ldots,r\}$, 
\[\lambda_{k+n-r}(A) \leq \lambda_{k}(A_r) \leq \lambda_{k}(A).\]
\end{lemma}

The \emph{spectral radius} of a matrix $A=(a_{ij})\in \mathbb{C}^{n\times n}$ is defined as $\rho(A):=\max\{|\lambda_i| : \lambda_i \text{ is an eigenvalue of } A \}$.  

The multiset of all eigenvalues of $A\in \mathbb{R}^{n \times n}$, denoted by $\sigma(A)$, is called the \emph{spectrum} of $A$.  The next lemma is often called the Ger\v{s}gorin disc theorem.
\begin{lemma}[\cite{MR1084815}, Theorem 6.1.1]\label{GDThm}  Suppose $A=(a_{ij})\in \mathbb{R}^{n\times n}$.  Then
\[ \sigma(A)\subseteq \bigcup_{i=1}^n \Big\{z\in\mathbb{C} : |z-a_{ii}|\leq \sum_{\substack{j=1\\j\neq i}}^n |a_{ij}|\Big\}.\]
\end{lemma}

\section{Adjacency Eigenvalues}\label{AdjEigenvaluesOH}

As mentioned before, if $G$ is an oriented hypergraph, then its adjacency matrix $A(G)$ is symmetric.  Therefore, $A(G)$ has real eigenvalues.  In this section we study the adjacency eigenvalues in relation to the structure of $G$.

The next lemma implies that a vertex-switching class has a single adjacency spectrum.  This is immediate from Lemma \ref{OHLAHSwitchingSimilarityTrans}.  Two oriented hypergraphs $G_1$ and $G_2$ are \emph{cospectral} if the adjacency matrices $A(G_1)$ and $A(G_2)$ have the same spectrum.  The following lemma also states that vertex-switching is a method for producing cospectral oriented hypergraphs.

\begin{lemma}\label{OHAdjEigenvaluesSWITCHIN} Let $G_1$ and $G_2$ both be oriented hypergraphs.  If $G_1\sim G_2$, then $G_1$ and $G_2$ are cospectral.
\end{lemma}

The spectral radius of the adjacency matrix of an oriented hypergraph $G$ is related to the number of adjacencies in $G$.  This is a generalization of a similar result known for graphs \cite[Proposition 1.1.1]{MR2571608}.
\begin{theorem}\label{OHSpecRadiusandMaxDegree}  Let $G$ be an oriented hypergraph.  Then
\[ \rho(A(G)) \leq \max_i \NumAdj(v_i).\]
\end{theorem}
\begin{proof}
The proof is inspired by the version for graphs \cite[Proposition 1.1.1]{MR2571608}.  Let $\mathbf{x}=(x_1,\ldots,x_n)\in \mathbb{R}^n$ be an eigenvector of $A(G)$ with associated eigenvalue $\lambda$.  The $i^\text{th}$ entry in the equation $A(G)\mathbf{x}=\lambda\mathbf{x}$ is
\[ \lambda x_i = \sum_{v_j\in V} \sum_{\substack{e\in E\\ v_i,v_j\in e}} \sgn_e(v_i,v_j) x_j .\]
Let $|x_m|=\max_{k}|x_k|\neq 0$.  Then,
\[ |\lambda||x_m| \leq  \sum_{v_j\in V} \sum_{\substack{e\in E\\ v_m,v_j\in e}} |x_j| \leq \max_i \NumAdj(v_i) |x_m| .\]
The result follows.
\end{proof}

The \emph{number of positive adjacencies containing $v_j$} is $\NumAdj^{+}(v_j):=|\{(e,\{v_j,\allowbreak v_k\})\in\mathcal{A}: \sgn_e(v_j,v_k)=+1\}|$.  The \emph{number of negative adjacencies containing $v_j$} is $\NumAdj^{-}(v_j):=|\{(e,\{v_j,v_k\})\in\mathcal{A}: \sgn_e(v_j,v_k)=-1\}|$.  Notice that $\NumAdj(v_j)=\NumAdj^{+}(v_j)+\NumAdj^{-}(v_j)$.  If $G$ is a simple linear 2-uniform oriented hypergraph, then $\NumAdj(v_j)=\NumAdj^{+}(v_j)+\NumAdj^{-}(v_j)$ is the same as $d_j= d_j^{+}+d_j^-$ (where $d_j^{+}$ and $d_j^{-}$ are the number of positive and negative edges incident to $v_j$), as known for signed graphs.  The \emph{net number of adjacencies containing $v_j$} is $\NumAdj^{\pm}(v_j):=\NumAdj^{+}(v_j)-\NumAdj^{-}(v_j)$.

The following adjacency eigenvalue bounds depend on the adjacency signs of an oriented hypergraph $G=(H,\sigma)$.  This is particularly interesting since the inequalities are not solely determined by the underlying hypergraph $H$.  Inequality \eqref{OHNeqAB1gain} is a generalization of a lower bound for the largest adjacency eigenvalue of an unsigned graph attributed to Collatz and Sinogowitz \cite{MR0087952}.

\begin{theorem}\label{OHAdjacencyBounds1} Let $G=(H,\sigma)$ be an oriented hypergraph.  Then
\begin{equation}\label{OHNeqAB1gain}
\lambda_{n}(A(G)) \leq \frac{1}{n} \sum_{j=1}^n \NumAdj^{\pm}(v_j)  \leq \lambda_{1}(A(G)).
\end{equation}
\end{theorem}

{\em Proof.}   The proof method is similar to \cite[Theorem 3.2.1]{MR2571608} and \cite[Theorem 8.1.25]{MR2571608}.
For brevity, we write $A$ for $A(G)$. Let $\mathbf{j}:=(1,\ldots,1)\in\mathbb{R}^n$.  Let $M_k={\bf j}^{\text{T}} A^k {\bf j}$.  From Lemma \ref{RRThm} the following is clear: 
\[ (\lambda_{n}(A))^k \leq M_k/\mathbf{j}^{\text{T}}\mathbf{j} \leq (\lambda_{1}(A))^k.  \]
We will use the equation:
\begin{align}
A{\bf j}&= \Big(\sum_{j=1}^n \sum_{e\in E}\sgn_e(v_1,v_j),\ldots,\sum_{j=1}^n \sum_{e\in E}\sgn_e(v_n,v_j)\Big)\notag\\
&=(\NumAdj^{\pm}(v_1),\ldots, \NumAdj^{\pm}(v_n)).\label{OHnetdegeq}
\end{align}
We compute $M_1$; thus, making inequality \eqref{OHNeqAB1gain} true.
\[ M_1 = {\bf j}^{\text{T}} A {\bf j}={\bf j}^{\text{T}}(\NumAdj^{\pm}(v_1),\ldots, \NumAdj^{\pm}(v_n)) = \sum_{j=1}^n \NumAdj^{\pm}(v_j).\cvd\]

Better bounds can be found by computing $M_k$ for larger $k$ values, as was done for graphs with $k=2$ by Hoffman \cite{MR0140441}.  

The adjacency eigenvalues of an oriented hypergraph $G$ bound the adjacency eigenvalues of the weak vertex-deletion $G\backslash v$.  This result is in some sense a generalization of similar bounds known for the adjacency eigenvlues of a graph $G$ and the adjacency eigenvalues of the vertex-deletion $G\backslash v$ \cite[Theorem 1.2.6]{liu2000matrices}.  However, as explained in the background section, the definition of weak vertex-deletion is different than the vertex-deletion from graph theory.

\begin{theorem}\label{OHAdjacencyInterlacing} Let $G$ be an oriented hypergraph, and let $v$ be some vertex of $G$. Then
\[\lambda_{k+1}(A(G)) \leq \lambda_{k}(A(G\backslash v)) \leq \lambda_{k}(A(G))\text{ for all }k \in \{1,\ldots,n-1\}.\]
\end{theorem}
\begin{proof} 
In the weak vertex-deletion $G\backslash v$, the only incidences that are removed from $G$ are those that contain the vertex $v$.  So the only adjacencies that are removed in the weak vertex-deletion are those that contain $v$.  Thus, the adjacency matrix $A(G\backslash v)$ can be obtained from $A(G)$ by deleting both rows and columns corresponding to the vertex $v$.  This shows that $A(G\backslash v)$ is an $(n-1)\times (n-1)$ principle submatrix of the adjacency matrix $A(G)$.  The result follows from Lemma \ref{interlacinglemma}.
\end{proof}

\section{Laplacian Eigenvalues}\label{LapEigenvaluesOHLapEigenvaluesOH}

For an oriented hypergraph $G$, the Laplacian matrix $L(G)$ is symmetric by definition, and hence, has real eigenvalues.  Moreover, Lemma \ref{OHLapIncidenceRelation} says that $L(G)$ is positive semidefinite, and therefore, $L(G)$ has nonnegative eigevalues.

The next lemma implies that a vertex-switching class has a single Laplacian spectrum.  This is immediate from Lemma \ref{OHLAHSwitchingSimilarityTrans}.  Two oriented hypergraphs $G_1$ and $G_2$ are \emph{Laplacian cospectral} if the Laplacian matrices $L(G_1)$ and $L(G_2)$ have the same spectrum.  The following lemma also shows that vertex-switching is a method for producing Laplacian cospectral oriented hypergraphs.

\begin{lemma}\label{OHVertSwitchingClassLapSpectrum} Let $G_1$ and $G_2$ both be oriented hypergraphs.  If $G_1\sim G_2$, then $G_1$ and $G_2$ are Laplacian cospectral.
\end{lemma}

The two products $\Eta(G)\Eta(G)^{\text{T}}$ and $\Eta(G)^{\text{T}}\Eta(G)$ are matrices with the same \allowbreak nonzero eigenvalues. This means that an oriented hypergraph and its incidence dual have the same nonzero Laplacian eigenvalues.

\begin{corollary}\label{OHGandIncidenceDualLapSpectrum}
If $G$ is an oriented hypergraph, then $L(G)$ and $L(G^*)$ have the same nonzero eigenvalues.
\end{corollary}
\begin{proof} This follows from Lemma \ref{OHLapIncidenceRelation}.
\end{proof}

If the number of vertices and edges are the same in an oriented hypergraph $G$ (i.e., $n=m$), it is impossible to distinguish $G$ and $G^*$ from their Laplacian spectra.  We have already seen that a vertex-switching class has a Laplacian spectrum (see Lemma \ref{OHVertSwitchingClassLapSpectrum}), but vertex-switching an oriented hypergraph $G$ will produce an oriented hypergraph that is usually very different from the incidence dual $G^*$.  Hence, Corollary \ref{OHGandIncidenceDualLapSpectrum} provides a potential method for producing a Laplacian cospectral oriented hypergraph that does not belong to the vertex-switching class of $G$.  In fact, it provides a potential method for producing two Laplacian cospectral vertex-switching classes.  That is, it may be that $[G]\neq[G^*]$, but $[G]$ and $[G^*]$ are Laplacian cospectral.

\begin{example}\label{OHGandDualLaplacianCospectralExample} Consider the oriented hypergraph $G$ and its incidence dual $G^*$ in Figure \ref{OHGandDualLaplacianCospectral}.  
\begin{figure}[h!]
    \includegraphics[scale=0.6]{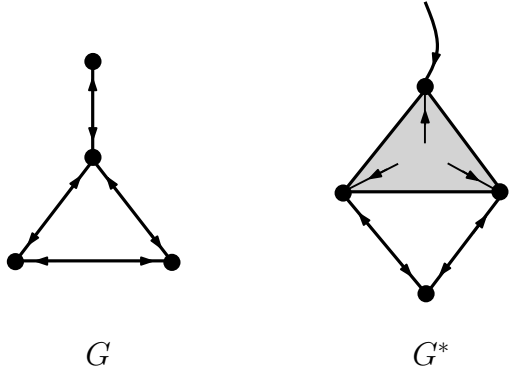}\centering
    \caption{An oriented hypergraph $G$ and its incidence dual $G^*$.}\label{OHGandDualLaplacianCospectral}
\end{figure}
The Laplacian matrices of $G$ and $G^*$ are
\begin{align*}
L(G)&=\Eta(G)\Eta(G)^{\text{T}} = \left[\begin{array}{rrrr}
2 & 1 & 1 & 0 \\
1 & 2 & 1 & 0 \\
1 & 1 & 3 & 1 \\
0 & 0 & 1 & 1
\end{array}\right], \\\intertext{and }
L(G^*)&=\Eta(G^*)\Eta(G^*)^{\text{T}} = \left[\begin{array}{rrrr}
2 & 1 & 1 & 1 \\
1 & 2 & 1 & 0 \\
1 & 1 & 2 & 1 \\
1 & 0 & 1 & 2
\end{array}\right].
\end{align*}
Both $L(G)$ and $L(G^*)$ have the same spectrum:
\[ \sigma(L(G))=\sigma(L(G^*))=\left\{ 1, 2, \frac{1}{2}(5-\sqrt{17}), \frac{1}{2}(5+\sqrt{17})\right\}.\]
Therefore, $G$ and $G^*$ are Laplacian cospectral.  Thus, we have produced Laplacian cospectral oriented hypergraphs which happen to be incidence duals, but are not in the same vertex-switching class since their underling hypergraphs are different.  \\

\noindent{\bf Question 1:} Are there other methods to produce Laplacian cospectral oriented hypergraphs other than vertex-switching or taking duals?  What if we only wanted the nonzero Laplacian eigenvalues of both oriented hypergraphs to be the same?
\end{example}

The following is a simplification of the quadratic form $\mathbf{x}^{\text{T}}L(G)\mathbf{x}$ for an oriented hypergraph $G$.

\begin{proposition}\label{OHLaplQuadFormProp}  Let $G=(H,\sigma)$ be an oriented hypergraph.  Suppose $\mathbf{x}=(x_1,x_2,\ldots,x_n)\in \mathbb{R}^n$.  Then
\[ \mathbf{x}^{\text{T}}L(G)\mathbf{x}=\mathbf{x}^{\text{T}}\Eta(G)\Eta(G)^{\text{T}}\mathbf{x} = \sum_{e \in E} \left(\sum_{v_k\in e} \sigma(v_k,e)x_k \right)^2.\] 
\end{proposition}
{\em Proof.}  Let $\mathbf{x}=(x_1,x_2,\ldots,x_n)\in \mathbb{R}^n$. Then
\[(\mathbf{x}^{\text{T}}\Eta(G))^{\text{T}}=\Eta(G)^{\text{T}}\mathbf{x} = \left(\sum_{k=1}^n \eta_{v_ke_1}x_k, \cdots,\sum_{k=1}^n \eta_{v_ke_m}x_k\right). \]
Therefore,
\begin{align*}
\mathbf{x}^{\text{T}} L(G) \mathbf{x} = \mathbf{x}^{\text{T}} \Eta(G) \Eta(G)^{\text{T}} \mathbf{x}&= (\mathbf{x}^{\text{T}} \Eta(G))(\mathbf{x}^{\text{T}} \Eta(G))^{\text{T}} \\
&= \sum_{t=1}^m\left( \sum_{k=1}^n \eta_{v_k e_t} x_k \right)^2\\
&= \sum_{e\in E}\left( \sum_{v_k\in e} \sigma(v_k,e) x_k \right)^2.\cvd
\end{align*}

For a signed graph, the incidence matrix relation $\eta_{je}=-\eta_{ie}\sgn(e)$ provides a method for further simplification of the quadratic form in terms of edge signs.  Since there is no analogue of an edge sign for oriented hypergraphs, further simplification is difficult.

An edge in an oriented hypergraph is {\it uniformly oriented} if all incidences containing that edge have the same sign.  An oriented hypergraph is \emph{uniformly oriented} if all of its edges are uniformly oriented.  For example, all of the edges from both $G$ and $G^*$ in Figure \ref{OHGandDualLaplacianCospectral} are uniformly oriented, and thus, both $G$ and $G^*$ are uniformly oriented.  Notice that uniformly oriented hypergraphs do not need to have every incidence in the oriented hypergraph signed the same, as in Example \ref{OHGandDualLaplacianCospectralExample}.  Also, notice that the associated Laplacian matrices $L(G)$ and $L(G^*)$ in Example \ref{OHGandDualLaplacianCospectralExample} are nonnegative.

\begin{lemma}\label{nonnegLapOHReq} Let $G$ be a linear oriented hypergraph.  Then $L(G)$ is nonnegative if and only if all edges are uniformly oriented.
\end{lemma}
\begin{proof}
For a simple oriented hypergraph $G$ the $(i,j)$-entry of $L(G)$ can be written as $l_{ij}=\sum_{e\in E} \eta_{ie}\eta_{je}$, by Lemma \ref{OHLapIncidenceRelation}.  The linear assumption, that is, the assumption that no two adjacent vertices are incident to more than one common edge, restricts the sum 
$\sum_{e\in E} \eta_{ie}\eta_{je}$ to have at most one nonzero term.  Therefore, $l_{ij}$ is either 0 or is exactly $\eta_{ie}\eta_{je}$ for some edge $e$ incident to $v_i$ and $v_j$.  Now
\begin{align*}
\eta_{ie}\eta_{je}\geq 0 &\iff [\eta_{ie}\geq 0\text{ and }\eta_{je}\geq 0]\text{ or }[\eta_{ie}\leq 0\text{ and }\eta_{je}\leq 0].
\end{align*}
Since this statement must be true for all vertices incident to a fixed edge $e$ it follows that all incidences containing edge $e$ have the same sign (or are otherwise 0).  This ensures $L(G)$ is nonnegative if and only if all edges are uniformly oriented. 
\end{proof}

For an oriented hypergraph $G=(H,\sigma)$ let $\mathcal{U}(G)$ be the set of all uniformly oriented hypergraphs with the same underlying hypergraph $H$ as $G$.

Hou, Li and Pan showed that the Laplacian spectral radius of the all negative signed graph provides an upper bound on the Laplacian spectral radius of all signed graphs with the same underlying graph \cite{MR1950410} .  For readers familiar with the signless Laplacian, this is equivalent to saying that the signless Laplacian spectral radius of a graph $\Gamma$ provides an upper bound on the Laplacian spectral radius of all signed graphs with underlying graph $\Gamma$ .  The signed graph result generalizes the same result known for graphs \cite{Shu2002123,MR2401311}.  Here we state a generalization to oriented hypergraphs.  It turns out that for oriented hypergraphs, the analogous structure of the all negative signed graph is a uniformly oriented hypergraph.

\begin{theorem}\label{UnivLapUBOrientedHypergraphs} Let $G=(H,\sigma)$ be a linear oriented hypergraph. Then for every $U \in \mathcal{U}(G)$,  \[ \lambda_1(L(G))\leq \lambda_1(L(U)).\]
\end{theorem}
\begin{proof}
The use of the quadratic form is inspired by the signed graphic proof in \cite[Lemma 3.1]{MR1950410}.  Let $G=(H,\sigma_G)$ and let $U=(H,\sigma_U)$ for some $U \in \mathcal{U}(G)$.  Let $\mathbf{x}=(x_1,\ldots,x_n)\in\mathbb{R}^n$ be a unit eigenvector of $L(G)$ with corresponding eigenvalue $\lambda_1(L(G))$.  By Proposition \ref{OHLaplQuadFormProp}:
\[\lambda_1(L(G)) = \mathbf{x}^{\text{T}}L(G)\mathbf{x} = \sum_{e \in E} \left( \sum_{v_k\in e} \sigma_G(v_k,e)x_k \right)^2
\leq \sum_{e \in E} \left(\sum_{v_k\in e} |x_k| \right)^2. \]
Since $\mathbf{x}=(x_1,x_2,\ldots,x_n)$ is a unit vector, $\mathbf{y}=(|x_1|,|x_2|,\ldots,|x_n|)$ is also a unit vector.  Hence,
\[\sum_{e \in E} \left(\sum_{v_k\in e} |x_k| \right)^2 = \sum_{e \in E} \left(\sum_{v_k\in e} y_k \right)^2 \leq \max_{\mathbf{z}^{\text{T}}\mathbf{z}=1}\sum_{e \in E} \left(\sum_{v_k\in e} z_k \right)^2. \]
Since $U$ is uniformly oriented we may assume $\sigma_U(v_k,e)=\alpha_e \in \{+1,-1\}$ for all $v_k\in e$.  Now, by Proposition \ref{OHLaplQuadFormProp} and Lemma \ref{RRThm}:
\begin{align*} \lambda_1(L(U)) = \max_{\mathbf{z}^{\text{T}}\mathbf{z}=1} \mathbf{z}^{\text{T}}L(U)\mathbf{z} &= \max_{\mathbf{z}^{\text{T}}\mathbf{z}=1} \sum_{e \in E} \left( \sum_{v_k\in e} \sigma_U(v_k,e)z_k \right)^2\\
&= \max_{\mathbf{z}^{\text{T}}\mathbf{z}=1} \sum_{e \in E} \left( \sum_{v_k\in e} \alpha_e z_k \right)^2\\
&= \max_{\mathbf{z}^{\text{T}}\mathbf{z}=1} \sum_{e \in E} \alpha_e^2\cdot\left( \sum_{v_k\in e} z_k \right)^2\\
&= \max_{\mathbf{z}^{\text{T}}\mathbf{z}=1} \sum_{e \in E} 1\cdot\left( \sum_{v_k\in e} z_k \right)^2\\
&= \max_{\mathbf{z}^{\text{T}}\mathbf{z}=1} \sum_{e \in E}\left( \sum_{v_k\in e} z_k \right)^2.
\end{align*}
Therefore, $\lambda_1(L(G))\leq \lambda_1(L(U))$.
\end{proof}

\noindent {\bf Question 2:} When does equality hold in Theorem \ref{UnivLapUBOrientedHypergraphs}?  If Hou, Li and Pan's result for signed graphs further generalizes to oriented hypergraphs, then equality holds if and only if $G$ is connected and vertex-switching equivalent to $U$.\\

Just like the adjacency eigenvalues, the Laplacian eigenvalues of an oriented hypergraph can be related to underlying structural parameters.  The following result generalizes 
the same upper bound known for the Laplacian spectral radius of a graph and the signless Laplacian spectral radius of a graph \cite{MR2401311}.

\begin{theorem} Let $G$ be an oriented hypergraph.  Then
\[ \lambda_1(L(G))\leq \max_{i} \{ d_i +\NumAdj(v_i) \}.\]
\end{theorem}
\begin{proof} By Lemma \ref{GDThm}, it is clear that for every $i\in\{1,\ldots,n\}$,
\[ |\lambda_1(L(G))|-|l_{ii}|\leq \big||\lambda_1(L(G))|-|l_{ii}|\big|\leq |\lambda_1(L(G))-l_{ii}|\leq \sum_{\substack{j=1\\j\neq i}}^n |l_{ij}|.\]
Therefore, 
\[|\lambda_1(L(G))|\leq |l_{ii}|+\sum_{\substack{j=1\\j\neq i}}^n |l_{ij}|\leq \max_{i}\Big\{ |l_{ii}|+\sum_{\substack{j=1\\j\neq i}}^n |l_{ij}|\Big\}
\leq \max_{i} \{ d_i +\NumAdj(v_i) \}.\]
Since $L(G)$ is positive semidefinite, $|\lambda_1(L(G))|=\lambda_1(L(G))$, and the result follows.
\end{proof}

To obtain a relationship between the Laplacian eigenvalues of an oriented hypergraph $G$ and the weak vertex-deletion $G\backslash v$ we will use the effect of weak vertex-deletion on the incidence matrix.  The same is also done for weak edge-deletion. Rusnak uses these results in his thesis \cite{OrientedHypergraphsThesisRusnak}, but are not formally stated.

\begin{lemma}\label{OHWeakVertDelIncidence} Let $G$ be an oriented hypergraph.
\begin{enumerate}
\item For any vertex $v$, $\Eta(G\backslash v)$ can be obtained from $\Eta(G)$ by deleting the row of $\Eta(G)$ corresponding to vertex $v$.  
\item For any edge $e$, $\Eta(G\backslash e)$ can be obtained from $\Eta(G)$ by deleting the column of $\Eta(G)$ corresponding to edge $e$.   
\end{enumerate}
\end{lemma}
\begin{proof}
For the proof of (1) recall that the weak vertex-deletion $G\backslash v$ will result in deleting $v$ from the vertex set, removing $v$ from every edge containing $v$, and deleting all incidences containing $v$.  Now by definition the incidence matrix $\Eta(G\backslash v)$ will have size $(|V|-1)\times |E|=(n-1)\times m$, and its entries are exactly the orientations assigned to the individual incidences of $G\backslash v$ or 0 otherwise.  The entries of $\Eta(G\backslash v)$ are identical to that of $\Eta(G)$, except that, since the weak vertex-deletion of $v$ removes all incidences of $G$ containing $v$, there is no row corresponding to $v$ in $\Eta(G\backslash v)$.  The result follows.

To prove (2) recall that the weak edge-deletion $G\backslash e$ will result in deleting $e$ from the edge set and removing all incidences containing $e$.  Now by definition the incidence matrix $\Eta(G\backslash e)$ will have size $|V|\times (|E|-1)=n\times (m-1)$, and its entries are exactly the orientations assigned to the individual incidences of $G\backslash e$ or 0 otherwise.  The entries of $\Eta(G\backslash e)$ are identical to that of $\Eta(G)$, except that, since the weak edge-deletion of $e$ removes all incidences of $G$ containing $e$, there is no column corresponding to $e$ in $\Eta(G\backslash e)$.  The result follows.
\end{proof}

Similar to the adjacency eigenvalue relationship presented in Theorem \ref{OHAdjacencyInterlacing}, the Laplacian eigenvalues of an oriented hypergraph $G$ bound the Laplacian eigenvalues of the weak vertex-deletion $G\backslash v$.  This Laplacian interlacing relationship is in some sense a generalization of the bounds known for the Laplacian eigenvalues of a graph $G$ and the Laplacian eigenvalues of the vertex-deleted graph $G\backslash v$ \cite{1176.05047}, but again, the weak vertex-deletion is not exactly the same as vertex-deletion.

\begin{theorem}\label{OHLapVertInterlacing} Let $G$ be an oriented hypergraph, and let $v$ be some vertex of $G$. Then
\[\lambda_{k+1}(L(G)) \leq \lambda_{k}(L(G\backslash v)) \leq \lambda_{k}(L(G))\text{ for all }k \in \{1,\ldots,n-1\}.\]
\end{theorem}
\begin{proof} From Lemma \ref{OHWeakVertDelIncidence}, $\Eta(G\backslash v)$ is obtained from $\Eta(G)$ by deleting the row of $\Eta(G)$ corresponding to vertex $v$.  Therefore, $\Eta(G\backslash v)\Eta(G\backslash v)^{\text{T}}$ is a principle submatrix of $\Eta(G)\Eta(G)^{\text{T}}$.  By Lemma \ref{OHLapIncidenceRelation}, $L(G\backslash v)=\Eta(G\backslash v)\Eta(G\backslash v)^{\text{T}}$ and $L(G)=\Eta(G)\Eta(G)^{\text{T}}$.  The result follows from Lemma \ref{interlacinglemma}.
\end{proof}

There is also a relationship between the Laplacian eigenvalues of an oriented hypergraph $G$ and the Laplacian eigenvalues of the weak edge-deletion $G\backslash e$.  This result generalizes the same result for the Laplacian eigenvalues of a graph \cite{Mohar91thelaplacian}, the signless Laplacian eigenvalues of a graph \cite{MR2401311} and the Laplacian eigenvalues of a signed graph \cite{MR1950410}.

\begin{theorem}\label{OHLapEDGEInterlacing} Let $G$ be an oriented hypergraph, and let $e$ be some edge of $G$. Then
\[\lambda_{k+1}(L(G)) \leq \lambda_{k}(L(G\backslash e)) \leq \lambda_{k}(L(G))\text{ for all }k \in \{1,\ldots,n-1\}.\]
\end{theorem}
\begin{proof} The proof is the same as the signed graph proof \cite[Lemma 3.7]{MR1950410}.  From Lemma \ref{OHWeakVertDelIncidence}, $\Eta(G\backslash e)$ is obtained from $\Eta(G)$ by deleting the column of $\Eta(G)$ corresponding to edge $e$.  Therefore, $\Eta(G\backslash e)^{\text{T}}\Eta(G\backslash e)$ is a principle submatrix of\\ $\Eta(G)^{\text{T}}\Eta(G)$.  Also, both $\Eta(G)^T\Eta(G)$ and $\Eta(G)\Eta(G)^{\text{T}}$ have the same nonzero eigenvalues.  By Lemma \ref{OHLapIncidenceRelation}, $L(G\backslash e)=\Eta(G\backslash v)\Eta(G\backslash e)^{\text{T}}$ and $L(G)=\Eta(G)\Eta(G)^{\text{T}}$.  The result follows from Lemma \ref{interlacinglemma}.
\end{proof}

The relationship between the Laplacian eigenvalues of $G$,  $G\backslash v$ and $G\backslash e$ can be used to obtain Laplacian eigenvalue bounds.  The next theorem relates the largest Laplacian eigenvalue to the maximum degree of an oriented hypergraph.  This generalizes a signed graphic bound that appears in \cite{MR1950410}, which generalizes an unsigned graphic version in \cite[p.186]{MR2571608}. 

\begin{figure}[h!]
    \includegraphics[scale=0.65]{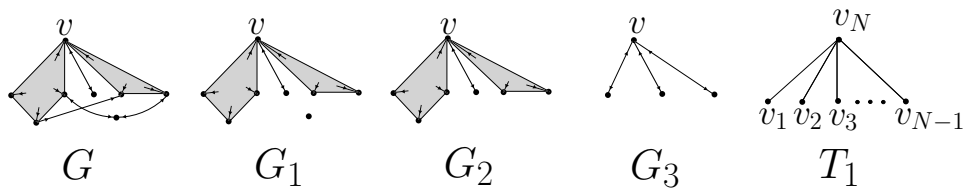}\centering
    \caption{An example of the deletion process described in the proof of Theorem \ref{OHLapLowerBoundDegree} of oriented hypergraphs $G$, $G_1$, $G_2$ and $G_3$ all with vertex $v$ having degree 3.  Also, the tree $T_1$ described in the same proof.}\label{ExampleDeletions}
\end{figure}

\begin{theorem}\label{OHLapLowerBoundDegree} Let $G$ be an oriented hypergraph where all edges have size at least 2.  Then
\[ \Delta+1 \leq \lambda_1(L(G)).\]
\end{theorem}
{\em Proof.}   The proof uses similar techniques to those of \cite[Theorem 3.10]{MR1950410}.
Let $v$ be a vertex in $G$ with $\text{deg}(v)=\Delta$.  See Figure \ref{ExampleDeletions} for a guiding example to the following general argument. Let $G_1$ be the oriented hypergraph obtained by weak edge-deletion of edges not incident to $v$ in $G$.  By repeated use of Lemma \ref{OHLapEDGEInterlacing}, $\lambda_1(L(G_1))\leq\lambda_1(L(G))$.  Let $G_2$ be the oriented hypergraph obtained by weak vertex-deletion of all isolated vertices in $G_1$.  By repeated use of Lemma \ref{OHLapVertInterlacing}, $\lambda_1(L(G_2))\leq\lambda_1(L(G_1))$.  For every edge $e$ of $G_2$ with $|e|\geq 3$, perform weak vertex-deletion on $|e|-2$ vertices of $e$ that have degree 1.  After all such weak vertex-deletions, pick one of the possible resulting oriented hypergraphs $G_3$.  By repeated use of Lemma \ref{OHLapVertInterlacing}, $\lambda_1(L(G_3))\leq\lambda_1(L(G_2))$.  Notice that $G_3$ is a 2-uniform oriented hypergraph.  The underlying (hyper)graph is the tree $T_1$ depicted in Figure \ref{ExampleDeletions} with $N=\Delta+1$.  By a simple calculation (see for example \cite[Lemma 5.6]{MR2900705}), $\lambda_1(L(T_1))=\Delta+1$.  It is clear that we can perform a vertex-switching on $G_3$ so that the adjacency signature is $+1$ on all adjacencies.  Since vertex-switching leaves the Laplacian eigenvalues unchanged by Lemma \ref{OHVertSwitchingClassLapSpectrum}, it is now clear that $\lambda_1(L(G_3))=\Delta+1$.  The result follows via the string of inequalities: \[\Delta+1=\lambda_1(L(G_3))\leq \lambda_1(L(G_2))\leq\lambda_1(L(G_1))\leq\lambda_1(L(G)).\cvd\]

Here we present Laplacian eigenvalue bounds which actually depend on the adjacency signature. 

\begin{theorem}\label{TOHthm1} Let $G=(H,\sigma)$ be an oriented hypergraph.  Then
\begin{equation}\label{OHNeqLB1gain}
\lambda_{n}(L(G)) \leq \frac{1}{n} \sum_{j=1}^n \big(d_j- \NumAdj^{\pm}(v_j)\big)  \leq \lambda_{1}(L(G)).
\end{equation}
\end{theorem}

{\em Proof.} 
The proof method is similar to \cite[Theorem 3.2.1]{MR2571608} and \cite[Theorem 8.1.25]{MR2571608} that was used for the adjacency eigenvalue bounds in Theorem \ref{OHAdjacencyBounds1}.  Let $\mathbf{j}:=(1,\ldots,1)\in\mathbb{R}^n$. Let $N_k:={\bf j}^{\text{T}} L(G)^k {\bf j}$.  From Lemma \ref{RRThm} the following is clear: 
\[ (\lambda_{n}(L(G)))^k \leq N_k/\mathbf{j}^{\text{T}}\mathbf{j} \leq (\lambda_{1}(L(G)))^k.  \]
Using Equation \eqref{OHnetdegeq} we will compute $N_1$; thus, making inequality \eqref{OHNeqLB1gain} true. 
\begin{align*}
N_1 = {\bf j}^{\text{T}} L(G) {\bf j}&={\bf j}^{\text{T}} (D(G)-A(G)) {\bf j} \\
&= {\bf j}^{\text{T}}\big((d_1,\ldots,d_n)-(\NumAdj^{\pm}(v_1),\ldots, \NumAdj^{\pm}(v_n))\big)\\
& = \sum_{j=1}^n \big(d_j- \NumAdj^{\pm}(v_j)\big).\cvd
\end{align*}

Better bounds can be found by computing $N_k$ for larger $k$ values.

\section{Hypergraph Spectra}\label{HypergraphSpectra}

A graph can be thought of as a signed graph with all edges labeled $+1$.  The oriented hypergraphic analogue of a signed graph with all edges labeled $+1$ is to have all adjacencies signed $+1$.  However, if the hypergraph has an edge of size greater than 2, there is no way to assign vertex-edge incidence labels (find $\sigma$) so that all adjacencies are signed $+1$.  Therefore, in general, there is no natural way to create an oriented hypergraph with all adjacencies signed $+1$.

However, it is possible to create a hypergraphic analogue of a signed graph with all edges signed $-1$.  To do this we need to assign vertex-edge incidences labellings (find $\sigma$) so that all adjacencies are signed $-1$.  This is accomplished if and only if all edges are uniformly oriented.  This is obvious since a $+1$ adjacency is formed when an edge is contained in two incidences that are oppositely signed.  Hence, a hypergraph $H$ can be thought of as an oriented hypergraph $G=(H,\sigma)$ where all edges are uniformly oriented.
All such uniformly oriented hypergraphs for a fixed $H$ produce the same adjacency and Laplacian matrices.  To further simplify things we can consider the two special cases where all edges are uniformly oriented the same way.  That is, not only do we require a uniformly oriented hypergraph, but one where every incidence is given the same sign.  In first case, all incidences of $H$ are assigned $+1$, so that all adjacencies are signed $-1$, producing the oriented hypergraph $+H=(H,+1)$.  In the second case, all incidences of $H$ are assigned $-1$, so that all adjacencies are signed $-1$, producing the oriented hypergraph $-H=(H,-1)$.  These two choices are the simplest possible orientations to pick and naturally define adjacency and Laplacian matrices.  

Therefore, to study hypergraph spectra one could use the following definitions.  The \emph{adjacency matrix of a hypergraph $H$} is defined as
\[ A(H) := A(H,+1) = A(H,-1).\]
The \emph{Laplacian matrix of a  hypergraph $H$} is defined as
\[ L(H) := L(H,+1) = L(H,-1).\]
These choices result in adjacency and Laplacian matrices that almost resemble the adjacency and Laplacian matrices developed by Rodr\'{i}guez \cite{MR1890984}.  However, since our adjacency entries will always be negative, our definition of the adjacency matrix is actually the negative of Rodr\'{i}guez's.  The Laplacian matrix can then be produced under this assumption.  For these special cases the results of  Rodr\'{i}guez  \cite{MR1890984} could naturally be generalized.

One advantage of these definitions for the adjacency and Laplacian matrices of a hypergraph is that there is no requirement for the hypergraph to be $k$-uniform, which has been the case for most hypergraph spectra definitions \cite{MR1235565,MR1405722,MR1325271,MR2842309}.  Another advantage is that these matrices are algebraically simpler to work with than hypermatrices, which provide an alternative version of hypergraph spectra for $k$-uniform hypergraphs \cite{MR2900714, extremalspectra1, nikiforov}.  Nikiforov states in \cite{nikiforov} that this version of hypergraph spectra ``is defined as a conditional maximum; thus, its usability in extremal problems is rooted in its very nature."  None of the bounds above involve extremal problems, but it would be interesting to see if these definitions could be used to solve such problems.  Cooper and Dutle's work \cite{MR2900714} covers a broad range of topics and includes structural bounds similar to the results above.  In particular, Theorems \ref{OHSpecRadiusandMaxDegree} and \ref{OHAdjacencyInterlacing} are related to Theorems 3.8 and 3.9 in \cite{MR2900714}, but the theorems presented here are valid for all hypergraphs (including oriented hypergraphs), and not just $k$-uniform hypergraphs.   Another advantage of the approach presented here is that the classic relationship between the incidence, adjacency and Laplacian matrices known for graphs and signed graphs is preserved to the hypergraph setting in Lemma \ref{OHLapIncidenceRelation}.

\section{Acknowledgments}
The author would like to thank the referee for the helpful comments for improving the quality of this paper.



\begin{thebibliography}{10}

\bibitem{MR1235565}
Fan R.~K. Chung.
\newblock The {L}aplacian of a hypergraph.
\newblock In {\em Expanding graphs ({P}rinceton, {NJ}, 1992)}, volume~10 of
  {\em DIMACS Ser. Discrete Math. Theoret. Comput. Sci.}, pages 21--36. Amer.
  Math. Soc., Providence, RI, 1993.

\bibitem{MR0087952}
Lothar Collatz and Ulrich Sinogowitz.
\newblock Spektren endlicher {G}rafen.
\newblock {\em Abh. Math. Sem. Univ. Hamburg}, 21:63--77, 1957.

\bibitem{MR2900714}
Joshua Cooper and Aaron Dutle.
\newblock Spectra of uniform hypergraphs.
\newblock {\em Linear Algebra Appl.}, 436(9):3268--3292, 2012.

\bibitem{MR2571608}
Drago{\v{s}} Cvetkovi{\'c}, Peter Rowlinson, and Slobodan Simi{\'c}.
\newblock {\em An {I}ntroduction to the {T}heory of {G}raph {S}pectra},
  volume~75 of {\em London Mathematical Society Student Texts}.
\newblock Cambridge University Press, Cambridge, 2010.

\bibitem{MR2401311}
Drago{\v{s}} Cvetkovi{\'c}, Peter Rowlinson, and Slobodan~K. Simi{\'c}.
\newblock Eigenvalue bounds for the signless {L}aplacian.
\newblock {\em Publ. Inst. Math. (Beograd) (N.S.)}, 81(95):11--27, 2007.

\bibitem{MR1405722}
Keqin Feng and Wen-Ch'ing~Winnie Li.
\newblock Spectra of hypergraphs and applications.
\newblock {\em J. Number Theory}, 60(1):1--22, 1996.

\bibitem{MR1325271}
Joel Friedman and Avi Wigderson.
\newblock On the second eigenvalue of hypergraphs.
\newblock {\em Combinatorica}, 15(1):43--65, 1995.

\bibitem{1176.05047}
Frank~J. Hall, Kinnari Patel, and Michael Stewart.
\newblock Interlacing results on matrices associated with graphs.
\newblock {\em J. Combin. Math. Combin. Comput.}, 68:113--127, 2009.

\bibitem{MR0140441}
A.~J. Hoffman.
\newblock On the exceptional case in a characterization of the arcs of a
  complete graph.
\newblock {\em IBM J. Res. Develop.}, 4:487--496, 1960.

\bibitem{MR1084815}
Roger~A. Horn and Charles~R. Johnson.
\newblock {\em Matrix {A}nalysis}.
\newblock Cambridge University Press, Cambridge, 1990.
\newblock Corrected reprint of the 1985 original.

\bibitem{MR1950410}
Yaoping Hou, Jiongsheng Li, and Yongliang Pan.
\newblock On the {L}aplacian eigenvalues of signed graphs.
\newblock {\em Linear Multilinear Algebra}, 51(1):21--30, 2003.

\bibitem{extremalspectra1}
Peter Keevash, John Lenz, and Dheuv Mubayi.
\newblock Spectral extremal problems for hypergraphs.
\newblock preprint:
  \href{http://arxiv.org/pdf/1304.0050v1.pdf}{arXiv:1304.0050}.

\bibitem{liu2000matrices}
Bolian Liu and Hong-Jian Lai.
\newblock {\em Matrices in {C}ombinatorics and {G}raph {T}heory}, volume~3 of
  {\em Network Theory and Applications}.
\newblock Kluwer Academic Publishers, Dordrecht, 2000.
\newblock With a foreword by Richard A. Brualdi.

\bibitem{MR2842309}
Linyuan Lu and Xing Peng.
\newblock High-ordered random walks and generalized {L}aplacians on
  hypergraphs.
\newblock In {\em Algorithms and models for the web graph}, volume 6732 of {\em
  Lecture Notes in Comput. Sci.}, pages 14--25. Springer, Heidelberg, 2011.

\bibitem{Mohar91thelaplacian}
Bojan Mohar.
\newblock The {L}aplacian spectrum of graphs.
\newblock In {\em Graph theory, combinatorics, and applications. {V}ol.\ 2
  ({K}alamazoo, {MI}, 1988)}, Wiley-Intersci. Publ., pages 871--898. Wiley, New
  York, 1991.

\bibitem{nikiforov}
Vladimir Nikiforov.
\newblock An analytic theory of extremal hypergraph problems.
\newblock preprint:
  \href{http://arxiv.org/pdf/1305.1073v2.pdf}{arXiv:1305.1073}.

\bibitem{MR2900705}
Nathan Reff.
\newblock Spectral properties of complex unit gain graphs.
\newblock {\em Linear Algebra Appl.}, 436(9):3165--3176, 2012.

\bibitem{ReffRusnak1}
Nathan Reff and Lucas~J. Rusnak.
\newblock An oriented hypergraphic approach to algebraic graph theory.
\newblock {\em Linear Algebra Appl.}, 437(9):2262--2270, 2012.

\bibitem{MR1890984}
J.~A. Rodr{\'{\i}}guez.
\newblock On the {L}aplacian eigenvalues and metric parameters of hypergraphs.
\newblock {\em Linear Multilinear Algebra}, 50(1):1--14, 2002.

\bibitem{OrientedHypergraphsThesisRusnak}
Lucas~J. Rusnak.
\newblock {\em Oriented hypergraphs}.
\newblock ProQuest LLC, Ann Arbor, MI, 2010.
\newblock Thesis (Ph.D.)--State University of New York at Binghamton.

\bibitem{Shu2002123}
Jin-Long Shu, Yuan Hong, and Wen-Ren Kai.
\newblock A sharp upper bound on the largest eigenvalue of the {L}aplacian
  matrix of a graph.
\newblock {\em Linear Algebra Appl.}, 347:123--129, 2002.

\bibitem{MR1120422}
Thomas Zaslavsky.
\newblock Orientation of signed graphs.
\newblock {\em European J. Combin.}, 12(4):361--375, 1991.

\end{thebibliography}

\end{document}